\documentclass[11pt]{amsart}

\usepackage{amssymb,amsmath,amsthm,verbatim,enumitem}
\usepackage{fullpage}
\usepackage{xcolor}
\usepackage{url}
\usepackage{mathrsfs}
\usepackage[all]{xy}
  \SelectTips{cm}{10}
  \everyxy={<2.5em,0em>:}
\usepackage{tikz}
\usepackage{picinpar} 

\usepackage{fancyhdr}
\usepackage{ifpdf}
  \ifpdf
    \usepackage{hyperref} 
  \else
    
    \newcommand{\href}[2]{#2}
  \fi

\theoremstyle{plain}
  \newtheorem{lemma}[equation]{Lemma}
  
  \newtheorem{theorem}[equation]{Theorem}

    \newtheorem{conjecture}[equation]{Conjecture}

\theoremstyle{definition}

\theoremstyle{remark}

\renewcommand{\thesection}{\arabic{section}}
\renewcommand{\theequation}{\thesection.\arabic{equation}}

 \DeclareFontFamily{U}{manual}{}
 \DeclareFontShape{U}{manual}{m}{n}{ <->  manfnt }{}
 \newcommand{\manfntsymbol}[1]{%
    {\fontencoding{U}\fontfamily{manual}\selectfont\symbol{#1}}}

\makeatletter
   \@addtoreset{section}{part}
   \@addtoreset{equation}{section}
   \@addtoreset{footnote}{section}

    {\hspace*{\fill}$\lrcorner$\endgraf\endgroup\end{trivlist}}
 
\makeatother

  \DeclareFontFamily{OT1}{pzc}{}
  \DeclareFontShape{OT1}{pzc}{m}{it}{<-> s * [1.100] pzcmi7t}{}
  \DeclareMathAlphabet{\mathpzc}{OT1}{pzc}{m}{it}

\newif\ifhascomments \hascommentstrue
\ifhascomments
  \newcommand{\dan}[1]{{\color{red}[[\ensuremath{\bigstar\bigstar\bigstar} #1]]}}
  \newcommand{\matt}[1]{{\color{red}[[\ensuremath{\spadesuit\spadesuit\spadesuit} #1]]}}
\else
  \newcommand{\dan}[1]{}
  \newcommand{\matt}[1]{}
\fi

\renewcommand{\>}{\rangle} 

\newcommand{\Qbar}{\overline{\mathbb{Q}}}
\renewcommand{\AA}{\mathbb{A}}

\DeclareMathOperator{\Aut}{\ensuremath{\mathcal{A}\kern-.125em\mathpzc{ut}}}

\newcommand{\bF}{\mathbb{F}}

\newcommand{\bZ}{\mathbb{Z}}

\newcommand{\dra}{\dashrightarrow}
\renewcommand{\emptyset}{\varnothing}

\DeclareMathOperator{\Endo}{\ensuremath{\mathcal{E}\kern-.125em\mathpzc{nd}}}

\newcommand{\OO}{\mathcal{O}}

\newcommand{\bP}{\mathbb{P}}

\DeclareMathOperator{\Hom}{\ensuremath{\mathcal{H}\kern-.125em\mathpzc{om}}}

\newcommand{\lra}{\longrightarrow}

\newcommand{\N}{\mathbb{N}}

\newcommand{\NN}{\mathbb N}
\renewcommand{\O}{\mathcal O}

\newcommand{\PP}{\mathbb{P}}

\newcommand{\QQ}{\mathbb Q}
\newcommand{\QQbar}{\overline{\mathbb Q}}
\newcommand{\RR}{\mathbb R}

\renewcommand{\setminus}{\smallsetminus}

\DeclareMathOperator{\spec}{Spec}

\newcommand{\X}{\mathcal{X}}

\newcommand{\Y}{\mathcal{Y}}

\newcommand{\ZZ}{\mathbb{Z}}

 \def\are[#1]{\ar[#1]^{\txt{\'et}}}
 \def\areh[#1]{\ar[#1]|{\txt{$H$-eq}}^{\txt{\'et}}}
 \def\ars[#1]{\ar@{->>}[#1]}
 \newcommand{\dplus}{\ar@{}[d]|{\mbox{$\oplus$}}}
 \newcommand{\dtimes}{\ar@{}[d]|{\mbox{$\times$}}}

\newcommand{\FF}{{\mathbb F}}

\newcommand{\cX}{\mathcal{X}}

\title{Dynamical Uniform Bounds for Fibers\\and a Gap Conjecture}

\author{Jason Bell}
\address{University of Waterloo \\
Department of Pure Mathematics \\
Waterloo, Ontario \\
N2L 3G1, Canada}
\email{jpbell@uwaterloo.ca}

\author{Dragos Ghioca}
\address{University of British Columbia\\
Department of Mathematics\\
Vancouver, BC\\
V6T 1Z2, Canada}
\email{dghioca@math.ubc.ca}

\author{Matthew Satriano}
\address{University of Waterloo \\
Department of Pure Mathematics \\
Waterloo, Ontario \\
N2L 3G1, Canada}
\email{msatrian@uwaterloo.ca}

\thanks{The authors were partially supported by Discovery Grants from the National Science and Engineering Research Council of Canada.}

\begin{document}

\begin{abstract}
We prove a uniform version of the Dynamical Mordell--Lang Conjecture for \'etale maps; also, we obtain a gap result for the growth rate of heights of points in an orbit along an arbitrary endomorphism of a quasiprojective variety defined over a number field. More precisely, for our first result, we assume $X$ is a quasi-projective variety defined over a field $K$ of characteristic $0$, endowed with the action of an \'etale endomorphism $\Phi$, and $f\colon X\lra Y$ is a morphism with $Y$ a quasi-projective variety defined over $K$. Then for any $x\in X(K)$, if for each $y\in Y(K)$, the set $S_y:=\{n\in \N\colon f(\Phi^n(x))=y\}$ is finite, then there exists a positive integer $N$ such that $\#S_y\le N$ for each $y\in Y(K)$. For our second result, we let $K$ be a number field, $f:X\dra \PP^1$ is a rational map, and $\Phi$ is an arbitrary endomorphism of $X$. If $\OO_\Phi(x)$ denotes the forward orbit of $x$ under the action of $\Phi$, then either $f(\OO_\Phi(x))$ is finite, or $\limsup_{n\to\infty} h(f(\Phi^n(x)))/\log(n)>0$, where $h(\cdot)$ represents the usual logarithmic Weil height for algebraic points.  
\end{abstract}

\maketitle

\section{Introduction}

As usual in algebraic dynamics, given a self-map $\Phi\colon X\lra X$ of a quasi-projective variety $X$, we denote by $\Phi^n$ the $n$-th iterate of $\Phi$. Given a point $x\in X$, we let $\OO_\Phi(x)=\{\Phi^n(x)\colon n\in\NN\}$ be the orbit of $x$. Recall that a point $x$ is periodic if there exists some $n\in\N$ such that $\Phi^n(x)=x$; a point $y$ is preperiodic if there exists $m\in\N$ such that $\Phi^m(y)$ is periodic. Our first result is the following.

\begin{theorem}
\label{thm:uniform-bound-fibers}
Let $X$ and $Y$ be quasi-projective varieties defined over a field $K$ of characteristic $0$, let $f\colon X\lra Y$ be a morphism defined over $K$, let $\Phi\colon X\lra X$ be an \'etale endomorphism, and let $x\in X(K)$.  
If $|\OO_\Phi(x)\cap f^{-1}(y)|<\infty$ for each $y\in Y(K)$, then there is a constant $N$ such that $$|\OO_\Phi(x)\cap f^{-1}(y)|<N$$ for each $y\in Y(K)$.
\end{theorem}

Theorem~\ref{thm:uniform-bound-fibers} offers a uniform statement for the Dynamical Mordell--Lang Conjecture. Indeed, the Dynamical Mordell--Lang Conjecture (see \cite{GT-JNT, DML-book}) predicts the following: given a quasi-projective variety $X$ defined over a field $K$ of characteristic $0$, endowed with an endomorphism $\Phi$, for any point $x\in X(K)$ and any subvariety $V\subset X$, the set 
$$S(X,\Phi,V,x):=\{n\in\N\colon \Phi^n(x)\in V(K)\}$$ 
is a finite union of arithmetic progressions $\{ak+b\colon k\in\N\}$ for some suitable integers $a$ and $b$, where the case $a=0$ yields a singleton instead of an infinite arithmetic progression. 

In particular, assuming $x$ is not preperiodic, if $V\subset X$ contains no periodic positive dimensional subvariety intersecting the orbit of $x$, then the Dynamical Mordell--Lang Conjecture predicts that $V$ intersects the orbit of $x$ in finitely many points, see \cite[\S3.1.3]{DML-etale}. The Dynamical Mordell--Lang Conjecture is still open in its full generality, though several partial results are known; for a full account of the known results prior to 2016, see \cite{DML-book}. One important case for which the Dynamical Mordell--Lang Conjecture is known is the case of \'etale endomorphisms, see \cite{DML-etale}. Our Theorem~\ref{thm:uniform-bound-fibers} yields a uniform statement for the Dynamical Mordell--Lang Conjecture in the case of \'etale endomorphisms, as follows.

Let $X$ be a quasi-projective variety defined over a field $K$ of characteristic $0$, endowed with an endomorphism $\Phi$. Let $\{X_y\}_{y\in Y}$ be an algebraic family of subvarieties of $X$ parametrized by some quasi-projective variety $Y$. Let $x\in X(K)$ be a non-preperiodic point with the property that its orbit under $\Phi$ meets each subvariety $X_y$ in finitely many points, i.e., no subvariety $X_y$ contains a periodic positive dimensional subvariety intersecting $\OO_\Phi(x)$. Then Theorem~\ref{thm:uniform-bound-fibers} proves that there exists a uniform upper bound $N$ for the number of points from the orbit $\OO_\Phi(x)$ on the subvarieties $X_y$ as $y$ varies in $Y(K)$. We believe the same statement would hold more generally (for an arbitrary endomorphism), as stated in the following uniform version of the Dynamical Mordell--Lang Conjecture.

\begin{conjecture} (Uniform Dynamical Mordell-Lang Conjecture) 
\label{conj:uniform DML}
Let $f\colon X\lra Y$ be a morphism of quasi-projective varieties defined over a field $K$ of characteristic $0$, let $\Phi\colon X\lra X$ be an endomorphism defined over $K$ and let $x\in X(K)$.  If $|\OO_\Phi(x)\cap f^{-1}(y)|<\infty$ for all $y\in Y(K)$, then there is a constant $N$ such that $$|\OO_\Phi(x)\cap f^{-1}(y)|<N$$ for all $y\in Y(K)$.
\end{conjecture}

Theorem~\ref{thm:uniform-bound-fibers} answers Conjecture~\ref{conj:uniform DML} in the case of \'etale endomorphisms. Furthermore, at the expense of replacing $\Phi$ by an iterate and also, replacing $X$ by $\Phi^\ell(X)$ for a suitable $\ell$, we see that Theorem~\ref{thm:uniform-bound-fibers} yields a positive answer for Conjecture~\ref{conj:uniform DML} for unramified endomorphisms $\Phi$ of a smooth variety $X$; thus, the conclusion of Theorem~\ref{thm:uniform-bound-fibers} applies to any endomorphism of a semiabelian variety $X$ defined over a field of characteristic $0$. 

One of the key lemmas from the proof of Theorem~\ref{thm:uniform-bound-fibers} (see Lemma~\ref{l:uniform-bound-piecewise-analytic}) provides the motivation for our next result (which is also motivated in its own right by the Dynamical Mordell--Lang Conjecture, as we will explain after its statement). 

\begin{theorem}
\label{thm:gaps}
Let $X$ be a  quasi-projective variety defined over $\Qbar$, let $\Phi\colon X\lra X$ be an endomorphism, and let $f\colon X\dra \bP^1$ be a rational function. Then for each $x\in X(\Qbar)$ with the property that the set $f(\OO_\Phi(x))$ is infinite, we have
\[
\limsup_{n\to\infty} \frac{h(f(\Phi^n(x)))}{\log(n)}>0,
\]
where $h(\cdot)$ is the logarithmic Weil height for algebraic numbers.
\end{theorem}

Note that if $X=\mathbb{A}^1$, the map $\Phi\colon X\lra X$ is given by $\Phi(x)=x+1$, and $f\colon X\hookrightarrow \bP^1$ is the usual embedding, then $h(f(\Phi^n(0)))=\log(n)$ for $n\in
\N$. This example shows that Theorem~\ref{thm:gaps} is, in some sense, the best possible. However, we believe that this gap result should hold more generally for rational self-maps.  Specifically, we make the following conjecture.

\begin{conjecture} (Height Gap Conjecture)
\label{conj:gaps}
Let $X$ be a  quasi-projective variety defined over $\Qbar$, let $\Phi\colon X\dra X$ be a rational self-map, and let $f\colon X\dra \bP^1$ be a rational function. Then for $x\in X(\Qbar)$ with the property that $\Phi^n(x)$ avoids the indeterminacy locus of $\Phi$ for every $n\ge 0$, either $f(\OO_\Phi(x))$ is finite or $$\limsup_{n\to\infty} \frac{h(f(\Phi^n(x)))}{\log(n)}>0.$$
\end{conjecture} 

Theorem \ref{thm:gaps} proves this conjecture in the case of endomorphisms.  Many interesting number theoretic questions fall under the umbrella of the gap conjecture stated above.  As an example, we recall that a power series $F(x)\in \Qbar[[x]]$ is called $D$-\emph{finite} if it is the solution to a non-trivial homogeneous linear differential equation with rational function coefficients.  It is known that if $\sum_{n\geq0} a(n) x^n$ is a $D$-finite power series over a field of characteristic zero, then there is some $d\ge 2$, a rational endomorphism $\Phi\colon\mathbb{P}^d\dra \mathbb{P}^d$, a point $c\in \mathbb{P}^d$ and a rational map $f\colon \mathbb{P}^d\dra \mathbb{P}^1$ such that $a(n)=f\circ \Phi^n(c)$ for $n\ge 0$, see \cite[Section 3.2.1]{DML-book}. Heights of coefficients of $D$-finite power series have been studied independently, notably by van der Poorten and Shparlinski \cite{vdPS}, who showed a gap result holds in this context that is somewhat weaker than what is predicted by our height gap conjecture above; specifically, they showed that if $\sum_{n\geq0} a(n)x^n\in \Qbar[[x]]$ is $D$-finite and $$\limsup_{n\to\infty} \frac{a(n)}{\log\log(n)}=0,$$ then the sequence $\{a(n)\}$ is eventually periodic.  This was improved recently \cite{BNZ}, where it is shown that if $\limsup_{n\to\infty} \frac{a(n)}{\log(n)}=0$, then the sequence $\{a(n)\}$ is eventually periodic. We see this then gives additional underpinning to Conjecture \ref{conj:gaps}. Furthermore, with the notation as in Theorem~\ref{thm:gaps}, assume now that 
\begin{equation}
\label{limsup is zero}
\limsup_{n\to\infty} \frac{h\!\left(f(\Phi^n(x))\right)}{\log(n)}=0.
\end{equation}
Then Theorem~\ref{thm:gaps} asserts that Equation~\eqref{limsup is zero} yields that $f(\OO_\Phi(x))$ is finite. We claim that actually this means that the set $\{f(\Phi^n(x))\}_{n\in\mathbb{N}}$ is eventually periodic. Indeed, for each $m\in\mathbb{N}$, we let $Z_m$ be the Zariski closure of $\{\Phi^n(x)\}_{n\ge m}$. Then $Z_{m+1}\subseteq Z_m$ for each $m$ and thus, by the Noetherian property, we get that there exists  some $M\in\mathbb{N}$ such that $Z_{m}=Z_M$ for each $m\ge M$. So, there exists a suitable positive integer $\ell$ such that $\Phi^\ell$ induces an endomorphism of each irreducible component of $Z_M$; moreover, each irreducible component of $Z_M$ contains a Zariski dense set of points from the orbit of $x$. Furthermore, because $f(\OO_\Phi(x))$ is a finite set, we get that  $f$ must be constant on each irreducible component of $Z_M$ and thus, in particular, $f$ is constant on each orbit $\OO_{\Phi^\ell}(\Phi^r(x))$ for $r$ sufficiently large. Hence, Theorem~\ref{thm:gaps} actually yields that once Equation~\eqref{limsup is zero} holds, then $\{f(\Phi^n(x))\}_{n\in\mathbb{N}}$ is eventually periodic.  
 
It is important to note that one cannot replace $\limsup$ with $\liminf$ in Conjecture \ref{conj:gaps}, even in the case of endomorphisms.  To see this, consider the map $\Phi\colon\mathbb{A}^3\to\mathbb{A}^3$ given by $(x,y,z)\mapsto (yz, xz, z+1)$.
Then, letting $c=(0,1,1)$, it is easily shown by induction that for $n\ge 0$, we have 
\[
\Phi^{2n}(c)=(0, (2n)!, 2n+1)\quad\quad\textrm{and}\quad\quad\Phi^{2n+1}(c)=((2n+1)!,0,2n+2).
\]
Consequently, if $f\colon\mathbb{A}^3\to \mathbb{A}^1$ is given by $f(x,y,z)=x+1$, then we see that $f(\Phi^{2n}(c))=1$ and $f(\Phi^{2n+1}(c))=(2n+1)!+1$ for every $n\ge 0$, and so
\[
\liminf_{n\to\infty} \frac{h(f(\Phi^{n}(c)))}{\log(n)}=0, \quad\quad\textrm{while}\quad\quad \limsup_{n\to\infty} \frac{h(f(\Phi^n(c)))}{\log(n)}=\infty.
\]
Despite the fact that the conjecture does not hold when one replaces $\limsup$ with $\liminf$, we believe the following variant of Conjecture \ref{conj:gaps} holds:

\begin{conjecture}
\label{conj:gaps-dense}
Let $X$ be an irreducible quasi-projective variety defined over $\Qbar$, let $\Phi\colon X\dra X$ be a rational self-map, and let $f\colon X\dra \bP^1$ be a non-constant rational function. Let $x\in X(\Qbar)$ with the property that $\Phi^n(x)$ avoids the indeterminacy locus of $\Phi$ for every $n\ge 0$, and further suppose that $\OO_\Phi(x)$ is Zariski dense in $X$. Then $$\liminf_{n\to\infty} \frac{h(f(\Phi^n(x)))}{\log(n)}>0.$$
\end{conjecture} 


We point out that, if true, this would be a powerful result and would imply the Dynamical Mordell--Lang conjecture for rational self-maps when we work over a number field.  To see this, let $Z$ be a quasi-projective variety defined over $\Qbar$, let $\Phi\colon Z\dra Z$ be a rational self-map, $Y$ be a subvariety of $Z$, and suppose that the orbit of $x\in Z(\QQbar)$ avoids the indeterminacy locus of $\Phi$. As before, denote by $Z_n$ the Zariski closure of $\{\Phi^j(x) \colon j\ge n\}$. Since $Z$ is a Noetherian topological space, there is some $m$ such that $Z_n=Z_m$ for every $n\ge m$. Letting $X=Z_m$, and replacing $Y$ with $Y\cap X$, it suffices to show that the conclusion to the Dynamical Mordell--Lang conjecture holds for the data $(X,\Phi, x, Y)$. We let $X_1,\ldots,X_d$ denote the irreducible components of $X$ and let $Y_i=Y\cap X_i$. Since $\Phi|_X$ is a dominant self-map, it permutes the components $X_i$, so there is some $b$ such that $\Phi^b(X_i)\subset X_i$ for each $i$.  Then if we let $x_1,\ldots ,x_d$ be elements in the orbit of $x$ with the property that $x_i\in X_i$, then it suffices to show that the conclusion to the statement of the Dynamical Mordell--Lang conjecture holds for the data $(X_i, \Phi^b, x_i, Y_i)$ for $i=1,\ldots,d$.  Then by construction, the orbit of $x_i$ under $\Phi^b$ is Zariski dense. We prove that either $\OO_{\Phi^b}(x_i)\subset Y_i$ or that  $\O_{\Phi^b}(x_i)$ intersects $Y_i$ finitely many times. If $Y_i=X_i$ or $Y_i=\emptyset$ then the result is immediate; thus we may assume without loss of generality that $Y_i$ is a non-empty proper subvariety of $X_i$. We pick a non-constant morphism $f_i\colon X_i\lra \mathbb{P}^1$ such that $f_i(Y_i)=1$. If $\Phi^{bn}(x_i)\in Y_i$, then $h(f(\Phi^{bn}(x_i)))=0$. Conjecture \ref{conj:gaps-dense} implies that this can only happen finitely many times, and so $\{n\colon \Phi^{bn}(x_i)\in Y_i\}$ is finite.

\section{Proof of our main results}
We recall the following definitions. The ring of strictly convergent power series $\mathbb{Q}_p\langle z\rangle$ is the collection of elements $P(z):=a_0+a_1 z+a_2 z^2 + \cdots \in \mathbb{Q}_p[[z]]$ such that $|a_n|_p\to 0$ as $n\to \infty$ and which thus consequently converge uniformly on $\mathbb{Z}_p$. The \emph{Gauss norm} is given by
$
|P(z)|_{\rm Gauss}:=\max_{n\geq0} |a_n|_p.
$ 
The ring $\ZZ_p\<z\>\subset\QQ_p\<z\>$ is the set of $P(z)$ with $|P(z)|_{\rm Gauss}\leq1$, i.e.~the set of $P$ with $a_i\in\ZZ_p$.

\begin{proof}[Proof of Theorem~\ref{thm:uniform-bound-fibers}.]
Clearly, we may reduce immediately (at the expense of replacing $\Phi$ by an iterate of it) to the case $X$ and $Y$ are irreducible.

A standard spreading out argument (similar to the one employed in the proof of \cite[Theorem~4.1]{DML-etale}) allows us to choose a model of $X$, $Y$, $f$, $\Phi$, and $x$ over an open subset $U\subseteq\spec R$, where $R$ is an integral domain which is a finitely generated $\ZZ$-algebra. In other words, $K$ is a field extension of the fraction field of $R$, we can find a map $\X\lra\Y$ over $U$, a section $U\lra\X$, and an \'etale endomorphism $\X\lra\X$ over $U$ which base change over $K$ to be $f\colon X\lra Y$, $x\colon\spec K\lra X$, and $\Phi\colon X\lra X$, respectively. After replacing $U$ by a possibly smaller open subset, we can assume $U=\spec R[g^{-1}]$ for some $g\in R$. Since $R[g^{-1}]$ is a finitely generated $\ZZ$-algebra, it is of the form $\ZZ[u_1,\dots,u_r]$. Applying \cite[Lemma 3.1]{generalized-SML}, we can find a prime $p\geq 5$ and an embedding $R[g^{-1}]$ into $\QQ_p$ which maps the $u_i$ into $\ZZ_p$. Base changing by the resulting map $\spec\ZZ_p\lra U$, we can assume $U=\spec\ZZ_p$. We will abusively continue to denote the map $\X\lra\Y$ by $f$, the \'etale endomorphism $\X\lra\X$ by $\Phi$, and the section $\spec\ZZ_p=U\lra\X$ by $x$. We let $\overline{\X}=\X\times_{\ZZ_p}\FF_p$, let $\overline{\Phi}\colon\overline{\X}\lra\overline{\X}$ be the reduction of $\Phi$, and let $\overline{x}\in\X(\FF_p)$ be the reduction of $x\in\X(\ZZ_p)$.

Notice that if $f(\Phi^n(x))=y$, then since $x$ extends to a $\ZZ_p$-point of $\X$, necessarily $y\in Y(K)$ extends to a $\ZZ_p$-point of $\Y$ as well. In particular, it suffices to give a uniform bound on the sets $\{n:f(\Phi^n(x))=y\}$ as $y$ varies through the elements $\Y(\ZZ_p)$.

To prove Theorem \ref{thm:uniform-bound-fibers}, we may replace $x$ by $\Phi^\ell(x)$ for some $\ell\in\NN$. Since $|\X(\FF_p)|<\infty$, we can therefore assume $\overline{x}$ is $\overline{\Phi}$-periodic, say of period $D$. It suffices to show that for each $1\leq i<D$, there is a uniform bound for $|\O_{\Phi^D(\Phi^i(x))}\cap f^{-1}(y)|$. 
So, we can replace $\Phi$ by $\Phi^D$, and hence assume $\overline{\Phi}(\overline{x})=\overline{x}$. Applying the $p$-adic Arc Lemma (see e.g., Remark 2.3 and Theorem 3.3 of \cite{DML-etale}) we can assume there are $p$-adic analytic functions $\phi_1,\dots,\phi_d\in\ZZ_p\<z\>$ such that in a $p$-adic analytic neighborhood, we have $$\Phi^n(x)=(\phi_1(n),\dots,\phi_d(n))\in\ZZ_p^d;$$ more precisely, letting $\phi(z):=(\phi_1(z),\dots,\phi_d(z))$, if $B\subset\X(\ZZ_p)$ is the $p$-adic ball of points whose reduction mod $p$ is $\overline{x}$, then there is an analytic bijection $\iota\colon B\lra\ZZ_p^d$, such that $\iota\!\left(\Phi^n(x)\right)=\phi(n)$.

Next, fix an embedding $\Y\subset\PP^r_{\ZZ_p}$, let $\{V_i\}_i$ be an open affine cover of $\Y$, and for each $i$, let $\{U_{ij}\}_j$ be an open affine cover of $f^{-1}(V_i)$. We can further assume that each $V_i$ is contained in one of the coordinate spaces $\AA^r_{\ZZ_p} \subset \PP^r_{\ZZ_p}$. Since $\X$ and $\Y$ are quasi-compact, we can assume the $\{U_{ij}\}_{i,j}$ and $\{V_i\}_i$ are finite covers. Then we can view $f|_{U_{ij}}\colon U_{ij} \lra V_i \subseteq \AA^r_{\ZZ_p}$ as a tuple of polynomials $(p_{ij0}, \dots, p_{ijr})$. Letting $P_{ijk}(z)=p_{ijk}\iota^{-1}\phi(z)$, we see $f|_{\O_{\Phi}(x)}$ is given by the following piecewise analytic function: $$f(\Phi^n(x))=(P_{ij0}(n), \dots, P_{ijr}(n))$$ whenever $\Phi^n(x)\in U_{ij}$.


It therefore suffices to prove that for each $i,j$, there exists $N_{ij}$ such that for all $(y_1,\dots,y_r)\in V_i(\ZZ_p)\subseteq\AA^r(\ZZ_p)$, the number of simultaneous roots of $P_{ijk}(z)-y_k$ (for $k=1,\dots,r$) is bounded by $N_{ij}$. In other words, we have reduced to proving the lemma below, where $S=\{n:\Phi^n(x)\in U_{ij}\}$ and $V=V_i(\ZZ_p)$.

\begin{lemma}
\label{l:uniform-bound-piecewise-analytic}
Let $r$ be a positive integer, let $V\subset\ZZ_p^r$, and let $S\subset\NN$ be an infinite subset. For each $1\leq k\leq r$, let $P_k\in\ZZ_p\<z\>$ 
and consider the function $P\colon S\lra\ZZ_p^r$ given by
\[
P(n):=(P_1(n),\dots,P_r(n)).
\]
Suppose the set $\{n\in S:P(n)=y\}$ is empty if $y\in \ZZ_p^r\setminus V$ and is finite if $y\in V$. Then there exists $N\geq0$ such that
\[
|\{n\in S:P(n)=y\}|\leq N
\]
for all $y \in V$.
\end{lemma}
\begin{proof}
We may assume $S$ is infinite since otherwise we can take $N=|S|$. We claim that $P_k(z)$ is not a constant power series for some $k$. Suppose to the contrary that $P_k(z)=c_k\in\ZZ_p$ for each $k$. If $y:=(c_1,\dots,c_r)\in\ZZ_p^r\setminus V$, then we can take $N=0$. If $y\in V$, then $\{n\in S:P(n)=y\}=S$ which is infinite, contradicting the hypotheses of the lemma.

We have therefore shown that some $P_k(z)$ is non-constant. Let $\mathcal{K}$ be the set of $k$ for which $P_k(z):=\sum_{m\geq0} c_{k,m}z^m$ is non-constant. Given any non-constant element $Q(z):=\sum_{m\geq0} c_mz^m$ of $\ZZ_p\<z\>$, let
\begin{equation}
\label{eqn:max-coeff-Gauss}
D(Q):=\max\{m:|c_m|=|Q|_{\rm Gauss}\}.
\end{equation}
Recall from Strassman's Theorem (see \cite{strassman} or \cite[Theorem 4.1, p.~62]{cassels}) that the number of zeros of $Q(z)$ is bounded by $D(Q)$. As a result, if $\alpha\in\ZZ_p$, then the number of zeros of $Q(z)-\alpha$ is bounded by $1+D(Q)$. Letting
\[
N:=1+\max_{k\in\mathcal{K}} D(P_k),
\]
we see then that for all $(y_1,\dots,y_r)\in\ZZ_p^r$, the number of simultaneous zeros of $P_1(z)-y_1$, $\dots$, $P_r(z)-y_r$ is bounded by $N$. In particular, $|\{n\in S:P(n)=y\}|\leq N$ for all $y\in V$.
\end{proof}
This concludes the proof of Theorem~\ref{thm:uniform-bound-fibers}.
\end{proof}

\begin{proof}[Proof of Theorem~\ref{thm:gaps}.] 
As before, at the expense of replacing $\Phi$ by an iterate, we may assume $X$ is irreducible. Furthermore, arguing as in the last paragraph of the introduction, we may assume $\O_\Phi(x)$ is Zariski dense.

Let $K$ be a number field such that $X$, $\Phi$, and $f$ are defined over $K$ and moreover, $x\in X(K)$. As proven in \cite{Schanuel}, there exists a constant $c_0>0$ such that for each real number  $N\ge 1$, there exist less than $c_0N^2$ algebraic points in $K$ of logarithmic height bounded above by $\log(N)$. So, there exists a constant $c_1>1$ such that for each real number $N\ge 1$, there are less than $c_1^N$ points in $K$ of logarithmic height bounded above by $N$.

Arguing as in the proof of Theorem~\ref{thm:uniform-bound-fibers}, we can find a suitable prime number $p$, a model $\cX$ of $X$ over some finitely generated $\bZ$-algebra $R$ which embeds into $\bZ_p$ such that the endomorphism $\Phi$ extends to an endomorphism of $\cX$, and a section ${\rm Spec}(\bZ_p)\lra \cX$ extending $x$; we continue to denote by $\Phi$ and $x$ the endomorphism of $\cX$ and the section ${\rm Spec}(\bZ_p)\lra \cX$, respectively. At the expense of replacing both $\Phi$ and $x$ by suitable iterates, we may assume the reduction of $x$ modulo $p$ (called $\overline{x}$) is fixed under the induced action of $\overline{\Phi}$ on the special fiber of $\cX$. Consider the $p$-adic neighborhood in $B\subset\X(\ZZ_p)$ consisting of all points whose reduction modulo $p$ is $\overline{x}$. Then there is an analytic isomorphism $\iota\colon B\to\ZZ_p^m$ so that in these coordinates
 $$\overline{x}=(0,\ldots, 0)\in \bF_p^m$$ 
 and 
$\Phi$ is given by $(x_1,\cdots, x_m)\mapsto (\phi_1(x_1,\dots, x_m),\cdots, \phi_m(x_1,\dots, x_m))$, where
$$\phi_i(x_1,\dots, x_m)\equiv \sum_{j=1}^m a_{i,j}x_j\pmod{p}$$
 for each $i=1,\dots, m$, for some suitable constants $a_{i,j}\in\bZ_p$ (for more details, see \cite[Section~11.11]{DML-book}). Applying \cite[Theorem~11.11.1.1]{DML-book} (see also the proof of \cite[Theorem~11.11.3.1]{DML-book}), there exists a $p$-adic analytic function $G\colon\bZ_p\lra \bZ_p^m$ such that for each $n\ge 1$, we have
\begin{equation}
\label{eq:p-adic approximation}
\|\Phi^n(x)-G(n)\|\le p^{-n},
\end{equation}
where for any point $(x_1,\dots, x_m)\in\bZ_p^m$, we let
$$\|(x_1,\dots, x_m)\|:=\max_{1\leq i\leq m} |x_i|_p.$$

As in the proof of Theorem~\ref{thm:uniform-bound-fibers}, let $V_1\simeq\AA^1$ and $V_2\simeq\AA^1$ be the standard affine cover of $\PP^1$, and let $\{U_{ij}\}$ be a finite open affine cover of $\X$ minus the indeterminacy locus of $f$ such that $f(U_{ij})\subset V_i\simeq\AA^1$. Let
\[
S_{ij}:=\{n:\Phi^n(x)\in U_{ij}\}.
\]
Since $f|_{U_{ij}}$ is given by a polynomial with $p$-adic integral coefficients, there exist $H_{ij}(z)\in\ZZ_p\<z\>$ such that
\[
f(G(n))=H_{ij}(n)
\]
whenever $n\in S_{ij}$. Notice that if $f(\Phi^n(x))=y$, then since $x$ extends to a $\ZZ_p$-point of $\X$, necessarily $y\in \PP^1(K)$ extends to a $\ZZ_p$-point of $\PP^1$ as well. Thus, we need only concern ourselves with roots of $H_{ij}(z)-t$ for $t\in\ZZ_p$.

\begin{lemma}
\label{l:constant-rate-of-return}
There is some choice of $i$ and $j$ with the following properties: 
\begin{enumerate}[label=$(\arabic*)$]
\item\label{constant-rate-of-return::infinite} $\{f(\Phi^n(x)):n\in S_{ij}\}$ is an infinite set,
\item\label{constant-rate-of-return::Banach-density} $\NN\smallsetminus S_{ij}$ has upper Banach density zero,
\item\label{constant-rate-of-return::return} there exists a constant $\kappa$ and a sequence $M_1<M_2<\dots$ such that 
\[
\#\{n\in S_{ij}:n\leq \kappa M_\ell\}\geq M_\ell.
\]
\end{enumerate}
In fact properties \ref{constant-rate-of-return::infinite} and \ref{constant-rate-of-return::Banach-density} hold for all $i$ and $j$.
\end{lemma}
\begin{proof}
We first prove property \ref{constant-rate-of-return::infinite} for all $i,j$. If $\{f(\Phi^n(x)):n\in S_{ij}\}=\{t_1,\dots,t_k\}$ is a finite set, then
\[
\O_\Phi(x)\subset(X\smallsetminus U_{ij})\cup\bigcup_{1\leq\ell\leq k}f^{-1}(t_\ell)
\]
which contradicts the fact that $\O_\Phi(x)$ is Zariski dense.

We next prove property \ref{constant-rate-of-return::Banach-density} for all $i,j$. Since $Z:=X\setminus U_{ij}$ is a closed subvariety, \cite[Corollary 1.5]{DML-noetherian} tells us $\NN\smallsetminus S$ is a union of at most finitely many arithmetic progressions and a set of upper Banach density zero. Since $\O_\Phi(x)$ is Zariski dense, $\NN\smallsetminus S$ cannot contain any non-trivial arithmetic progressions. Indeed, if there exists $0\leq b<a$ such that $\{an+b:n\in\NN\}\subset\NN\smallsetminus S$, then writing $\O_\Phi(x)=\bigcup_{0\leq\ell<a}\Phi^\ell\O_{\Phi^a}(x)$, we see $\Phi^\ell\O_{\Phi^a}(x)$ is Zariski dense for some $\ell$; applying a suitable iterate of $\Phi$, we see $\Phi^b\O_{\Phi^a}(x)$ is also Zariski dense, contradicting the fact that $\Phi^b\O_{\Phi^a}(x)$ is contained in the proper closed subvariety $Z$. Thus, $\NN\smallsetminus S$ has upper Banach density zero.

Next, we turn to property \ref{constant-rate-of-return::return}. Let $\mathcal{U}$ denote our set of affine patches $U_{ij}$, and let $\kappa=|\mathcal{U}|$. For each $M\geq2$, there must exist some element $g(M)\in\mathcal{U}$ and $0\leq n_1<n_2<\dots<n_M\leq (M-1)\kappa$ such that $\Phi^{n_\ell}(x)\in g(M)$ for $1\leq\ell\leq M$. Let $i$ and $j$ be such that $g(M)=U_{ij}$ for infinitely many $M\geq2$. With this choice of $i$ and $j$, by construction, there is a sequence $M_1<M_2<\dots$ such that $S_{ij}$ contains at least $M_\ell$ of the integers $0\leq n\leq(M_\ell-1)\kappa$. As a result,
\[
\#\{n\in S_{ij}:n\leq \kappa M_\ell\}\geq M_\ell
\]
finishing the proof of lemma.
\end{proof}

Let $i$ and $j$ be as in Lemma \ref{l:constant-rate-of-return}. For ease of notation, let $H(z)=H_{ij}(z)$ and $S=S_{ij}$. Since
\[
\limsup_{n\to\infty} \frac{h(f(\Phi^n(x)))}{\log(n)}\geq \limsup_{n\in S, n\to\infty} \frac{h(f(\Phi^n(x)))}{\log(n)}
\]
it suffices to show the latter is positive.


We split our proof now into two cases which are analyzed separately in Lemmas~\ref{lem:approx_1}~and~\ref{lem:approx_2}.  Before giving the proof of Lemma~\ref{lem:approx_1}, we recall that by the Weierstrass Preparation Theorem \cite[5.2.2]{non-arch-analysis}, 
if $P(z):=a_0+a_1 z+a_2 z^2+ \cdots \in \mathbb{Q}_p\langle z\rangle$ is nonzero and $D=D(P)$ as in (\ref{eqn:max-coeff-Gauss}), then $P(z)=Q(z)u(z)$ where $u(z)$ is a unit in $\mathbb{Q}_p\langle z\rangle$ with $|u(z)|_{\rm Gauss}=1$, and $Q(z)$ is a polynomial of degree $D$ whose leading coefficient has $p$-adic norm equal to $|P(z)|_{\rm Gauss}$. Combined with \cite[5.1.3 Proposition 1]{non-arch-analysis}, we see $u(z)=c+pu_0(z)$ with $|c|_p=1$ and $|u_0|_{\rm Gauss} \le 1$. In particular, $|u(n)|_p=1$ for all $n\in\NN$.

\begin{lemma}
\label{lem:approx_1}
If $H(z)$ is non-constant, then the conclusion of Theorem~\ref{thm:gaps} holds.
\end{lemma}

\begin{proof}[Proof of Lemma~\ref{lem:approx_1}.]
Writing $H(z)=a_0+a_1z+a_2z^2+\cdots\in\ZZ_p\<z\>$, there exists some $L\ge 1$ such that $|a_L|_p>|a_j|_p$ for all $j>L$. As proven in Lemma~\ref{l:uniform-bound-piecewise-analytic}, since $H(z)$ is not constant, there exists a uniform bound $C$ such that for each $t\in \mathbb{Z}_p$, the number of solutions to $H(z)=t$ is at most $C$. Furthermore, if $n$ is an element of $S$ such that $f(\Phi^n(x))=t$, then equation \eqref{eq:p-adic approximation} 
yields $$|H(n)-t|_p\le p^{-n}.$$ 
As mentioned above, by the Weierstrass Preparation Theorem, we can write
\[
H(z)-t= q_t(z)u_t(z)
\]
with $q_t(z)$ a polynomial of degree $D(H-t)\leq L$ and $u_t(z)$ a unit of Gauss norm $1$; moreover, the leading coefficient of $q_t(z)$ has $p$-adic norm equal to the Gauss norm of $H-t$. Hence, we can write
\[
q_t(z)=b_t(z-\beta_{1,t})\cdots (z-\beta_{D(H-t),t})
\]
with $b_t\in \mathbb{Q}_p$, the $\beta_{j,t}\in \QQbar_p$, and
\[
|b_t|_p=|H-t|_{\rm Gauss}\geq |a_L|_p.
\]
We have therefore bounded $|b_t|_p$ below independent of $t\in\ZZ_p$. As noted before the proof of the lemma, we know $|u_t(n)|_p=1$ for all $t\in\ZZ_p$ and $n\in\NN$. Hence, there is a constant $c_2>0$ (independent of $t$) such that for all $t\in\ZZ_p$, if $|H(n)-t|_p\le p^{-n}$ then there exists $1\leq j\leq D(H-t)$ such that
\[
|n-\beta_{j,t}|_p<c_2 p^{-\frac{n}{D(H-t)}}\leq c_2 p^{-\frac{n}{L}}.
\]
So, if $n_1,\ldots ,n_{L+1}$ are distinct elements of $S$ with $|H(n_i)-t|_p\le p^{-n_i}$ for $i=1,\ldots ,L+1$ then there exist $k_1,k_2$ with $k_1\neq k_2$ and $j$ such that $|n_{k_1}-\beta_{j,t}|_p<c_2 p^{-n_{k_1}/L}$ and $|n_{k_2}-\beta_{j,t}|_p<c_2 p^{-n_{k_2}/L}$.  Consequently, 
\[
|n_{k_1}-n_{k_2}|_p<c_2 p^{-\min(n_{k_1},n_{k_2})/L}.
\]
Hence, letting $|\cdot |$ be the usual Archimedean absolute value, we have that
\[
|n_{k_1}-n_{k_2}|>c_2 p^{\min(n_{k_1},n_{k_2})/L};
\]
therefore there exists a positive constant $c_3$ (independent of $t$, since both $L$ and $c_2$ are independent of $t$) such that for all $M\ge 1$ and all $t\in\bP^1(K)$,
\begin{equation}
\label{eq:c3-bound}
\#\{n\le M:n\in S \textrm{\ and\ } f(\Phi^n(x))=t\}\leq c_3\log(M).\footnote{In fact, we have a substantially better bound. Let $\exp^k$ denote the $k$-th iterate of the exponential function and let $L_p(M)$ be the smallest integer $k$ such that $\exp^k(p)>M$. Then $\#\{n\le M: n\in S \textrm{\ and\ } f(\Phi^n(x))=t\}\leq c_3 L_p(M)$, however we will not need this stronger bound.}
\end{equation}
As an aside, we note that this type of gap is similar to the one obtained for the Dynamical Mordell--Lang problem in \cite{gap-Compo}.

Now, let $\kappa$ be as in Lemma \ref{l:constant-rate-of-return}, and choose a constant $c_4>1$ such that
\begin{equation}
\label{eq:c43}
c_3\cdot \log(\kappa c_4^r)\cdot c_1^r<c_4^{r-1}
\end{equation}
for all sufficiently large $r\in\RR$, e.g.~we may take $c_4:=2c_1$. Let $M_1<M_2<\dots$ be as in Lemma \ref{l:constant-rate-of-return}, and let
\[
N_\ell=\lceil\log_{c_4}(M_\ell)\rceil.
\]
Property \ref{constant-rate-of-return::return} of Lemma \ref{l:constant-rate-of-return} implies
\begin{equation}
\label{eq:constant-rate-of-return-inequality}
\#\{n\leq\kappa c_4^{N_\ell}:n\in S\}\geq M_\ell>c_4^{N_\ell-1}.
\end{equation}

To conclude the proof, we show that for all $\ell$ sufficiently large, there exists some $n_\ell\le \kappa c_4^{N_\ell}$ with the property that $n_\ell\in S$ and $h(f(\Phi^{n_\ell}(x)))\ge N_\ell$. If this were not the case, then since there are less than $c_1^{N_\ell}$ algebraic numbers $t\in\bP^1(K)$ of logarithmic Weil height bounded above by $N_\ell$, by (\ref{eq:constant-rate-of-return-inequality}) there would be such an algebraic number $t$ with 
\[
\#\{n\le \kappa c_4^{N_\ell}:n\in S \textrm{\ and\ } f(\Phi^n(x))=t\}>\frac{c_4^{N_\ell-1}}{c_1^{N_\ell}}>c_3\log(\kappa c_4^{N_\ell})
\]
and this violates inequality (\ref{eq:c3-bound}). We have therefore proven our claim that for all $\ell$ sufficiently large, there exists a positive integer $n_\ell\leq \kappa c_4^{N_\ell}$ with $h(f(\Phi^{n_\ell}(x)))\ge N_\ell$. So, 
\[
\limsup_{n\to\infty}\frac{h(f(\Phi^n(x)))}{\log(n)} \geq \lim_{\ell\to\infty}\frac{N_\ell}{\log(\kappa)+N_\ell\log(c_4)}=\frac{1}{\log(c_4)}>0
\]
as desired in the conclusion of Theorem~\ref{thm:gaps}.
\end{proof}

\begin{lemma}
\label{lem:approx_2}
If $H(z)$ is a constant, then $\limsup_{n\to\infty}\frac{h(f(\Phi^n(x)))}{\log(n)}=\infty$.
\end{lemma}

\begin{proof}[Proof of Lemma~\ref{lem:approx_2}.]
By property \ref{constant-rate-of-return::infinite} of Lemma \ref{l:constant-rate-of-return}, we can find a sequence $n_1<n_2<\dots$ with the $n_i\in S$ such that
\[
f\!\left(\Phi^{n_{2k-1}}(x)\right)\neq f\!\left(\Phi^{n_{2k}}(x)\right) \textrm{\ \ and\ \ } \{n\in S:n_{2k-1}<n<n_{2k}\}=\varnothing
\]
for all $k\geq1$.

Let $t_0:=H(n)$ (for all $n\in\mathbb{N}$) and for each $i\ge1$, let $t_i:=f\!\left(\Phi^{n_i}(x)\right)$. Then using the Taylor expansion around $t_0$, we conclude that 
\[
|t_{i}-t_0|_p\le p^{-n_i}.
\]
So, $|t_{2k}-t_{2k-1}|_p\le p^{-n_{2k-1}}$ and since $t_{2k}\ne t_{2k-1}$, we have that for all $k\ge 1$,
\begin{equation}
\label{eq:large height}
h\!\left(t_{2k}-t_{2k-1}\right)=h\!\left((t_{2k}-t_{2k-1})^{-1}\right)\ge c_5n_{2k-1},
\end{equation}
for a constant $c_5$ depending only on the number field $K$ and on the particular embedding of $K$ into $\mathbb{Q}_p$ (for example, for the usual embedding of $K$ into $\mathbb{Q}_p$, we may take $c_5:=\frac{1}{2[K:\mathbb{Q}]}$). Inequality \eqref{eq:large height} yields that 
\begin{equation}
\label{eq:one height is large}
\max\{h(t_{2k}), h(t_{2k-1})\}\ge \frac{1}{2}(c_5 n_{2k-1} - \log(2))
\end{equation}
since $h(a+b)\le h(a)+h(b)+\log(2)$ for any $a,b\in\Qbar$.

First suppose that $\max\{h(t_{2k}), h(t_{2k-1})\}=h(t_{2k-1})$ for infinitely many $k$. Then consider a subsequence $k_1<k_2<\dots$ where $\max\{h(t_{2k_j}), h(t_{2k_j-1})\}=h(t_{2k_j-1})$. Letting $m_j=n_{2k_j-1}$, we see
\[
h\!\left(f\!\left(\Phi^{m_j}(x)\right)\right)\ge \frac{1}{2}(c_5 m_j - \log(2))
\]
which shows $\limsup_{n\to\infty}\frac{h(f(\Phi^n(x)))}{\log(n)}=\infty$.

Thus, we may assume that $\max\{h(t_{2k}), h(t_{2k-1})\}=h(t_{2k})$ for all $k$ sufficiently large. We claim that
\begin{equation}
\label{eq:even-term-subseq-limsup-infty}
\limsup_{k\to\infty}\frac{h(f(\Phi^{n_{2k}}(x)))}{\log(n_{2k})}=\infty.
\end{equation}
If this is not the case, then there is some $C'>0$ such that for all sufficiently large $k$, we have
\[
C'>\frac{h(f(\Phi^{n_{2k}}(x)))}{\log(n_{2k})}\geq \frac{1}{2\log(n_{2k})}(c_5 n_{2k-1} - \log(2)),
\]
where we have made use here of inequality (\ref{eq:one height is large}). In particular, there is a constant $C>1$ such that for all $k$ sufficiently large,
\begin{equation}
\label{eqn:big-gaps}
n_{2k}>C^{\hspace{0.1em}n_{2k-1}}.
\end{equation}
Recalling that $S$ does not contain any positive integers between $n_{2k-1}$ and $n_{2k}$, inequality (\ref{eqn:big-gaps}) implies that $\NN\smallsetminus S$ has positive upper Banach density. This contradicts property \ref{constant-rate-of-return::Banach-density} of Lemma \ref{l:constant-rate-of-return}, and so our initial assumption that $C'>\frac{h(f(\Phi^{n_{2k}}(x)))}{\log(n_{2k})}$ is incorrect. This proves equation (\ref{eq:even-term-subseq-limsup-infty}), and hence Lemma \ref{lem:approx_2}.
\end{proof}

Clearly, Lemmas~\ref{lem:approx_1} and \ref{lem:approx_2} finish the proof of Theorem~\ref{thm:gaps}.
\end{proof}


\end{document}